\documentclass{amsart}
\usepackage{physics,amsthm,amscd,amssymb,verbatim,epsf,amsmath,amsfonts,mathrsfs,graphicx,dirtytalk,pdfpages,mdframed,tikz-cd,relsize}
\usepackage[colorlinks=true,linkcolor=blue,citecolor=blue]{hyperref}

\usepackage{caption}
\usepackage{subcaption}
\usepackage[shortlabels]{enumitem}
\theoremstyle{plain}
\newtheorem{Thm}{Theorem}[section]
\newtheorem{Cor}[Thm]{Corollary}

\newtheorem{Lem}[Thm]{Lemma}
\newtheorem{Prop}[Thm]{Proposition}
\theoremstyle{definition}
\newtheorem{Def}[Thm]{Definition}

\newtheorem{remark}[Thm]{Remark}
\theoremstyle{remark}

\numberwithin{equation}{section}
\newcommand{\tp}{\text{P}}
\newcommand{\tq}{\text{Q}}
\newcommand{\dom}[1]{\text{Dom}(#1)}
\newcommand{\lsin}{L^2_{\sin \theta}(0,\pi)}
\newcommand{\lw}{L^2_{w}(a,b)}
\newcommand{\dsa}{D_\text{S.A.}}
\newcommand{\dmax}{D_\text{max}}

\date{\today}
\title{On the Convergence of Non-Integer \\ Linear Hopf Flow}
\author{Brendan Guilfoyle}
\address{Brendan Guilfoyle\\
          School of STEM\\
          Munster Technological University, Kerry\\
          Tralee\\
          Co. Kerry\\
          Ireland.}
\email{brendan.guilfoyle@mtu.ie}

\author{Morgan Robson}
\address{Morgan Robson\\
          School of STEM\\
          Munster Technological University, Kerry\\
          Tralee\\
          Co. Kerry\\
          Ireland.}
\email{morgan.robson@hotmail.com}
\begin{document}
\begin{abstract}
    The evolution of a rotationally symmetric surface by a linear combination of its radii of curvature equation is considered. It is known that if the coefficients form certain integer ratios the flow is smooth and can be integrated explicitly. 
    In this paper the non-integer case is considered for certain values of the coefficients and with mild analytic restrictions on the initial surface. 
    
    We prove that if the focal points at the north and south poles on the initial surface coincide, the flow converges to a round sphere. Otherwise the flow converges to a non-round Hopf sphere. Conditions on the fall-off of the astigmatism at the poles of the initial surface are also given that ensure the convergence of the flow.
    
    The proof uses the spectral theory of singular Sturm-Liouville operators to construct an eigenbasis for an appropriate space in which the evolution is shown to converge.

\end{abstract}


\maketitle
\section{Introduction and Results} \label{ch: Introduction}

Given a $C^2$-smooth surface $\mathcal{S}_0\subset \mathbb{R}^3$, consider the curvature flow problem of finding $\vec{X}:S^2\times [0,\infty)\to\mathbb{R}^3$
\begin{align}\label{eq: curvature flow equation}
    &\left(\pdv{\vec{X}}{t}\right)^\perp=(ar_1+br_2+c)\widehat{n},
    & {\mbox{ with }}
    &\vec{X}(S^2,0)=\mathcal{S}_0,
\end{align}
where $a,b,c\in\mathbb{R}$, $\vec{X}$ and $\hat{n}$ are the position and normal vectors of the evolving surface, and $r_1,r_2$ are the radii of curvature of the surface. In this paper we investigate the time evolution of a strictly convex topological 2-sphere with an axis of rotational symmetry under the above \textit{linear Hopf flow} \cite{gk21}. The flow is parabolic when $b>0$.

The literature on extrinsic curvature flows primarily concerns flows whose normal speed is a symmetric function in the radii of curvature \cite{andrews04}, for which flows by mean curvature, inverse mean curvature and powers of Gauss curvature are examples \cite{Brendle17} \cite{Daskalopoulos17} \cite{White19}. Symmetry is necessary for the normal speed to be well defined on general surfaces \cite{andrews20} as the radii of curvature may be exchanged via re-paramterisation.

The current work differs in that linear Hopf flow is not a symmetric curvature flow. Rotationally symmetric surfaces support a canonical labelling of the radii of curvature associated to the meridian and profile principal foliations of the surface, therefore are good surfaces with which to explore asymmetric curvature flows. Furthermore, rotational symmetry reduces the problem to one spatial variable. This extra tractability allows for a wider class of curvature flows to be considered e.g. \cite{Rengaswami21} \cite{Santaella20}. 




In \cite{gk21} the linear Hopf flow was completely solved when the \textit{flow slope} $-a/b$ takes one of the values
\begin{equation}\label{eq: quantised slope equ}
    -a/b=2n+3,
\end{equation}
with $n \in \mathbb{N}$. It was proven that the fate of an initial smooth sphere is entirely determined by the local geometry of its isolated umbilic points, in particular the order of vanishing of the difference between the radii of curvature at the poles: $s=r_2-r_1$. 

In what follows the focus is on \textit{non-integer} linear Hopf flow, i.e. the case in which $n$ as appears in equation (\ref{eq: quantised slope equ}) is non integer. In contrast to the integer case, non-round, smooth stationary solutions to the non-integer linear Hopf flow do not exist, and so the analysis is significantly complicated. 

To state the main results, let $\theta$ be the angle between the normal of $\mathcal{S}$ at a point and the axis of rotational symmetry and denote the second order Legendre differential operator for $\mu,\nu \in \mathbb{R}$ by $\mathscr{L}^\mu_\nu$ (given by equation (\ref{e:legop}) below). Let $\lsin$ be the space of square integrable functions with weight $\sin \theta$ on the interval $(0,\pi)$. 

\vspace{0.1in}
\begin{Thm}\label{thm: main thm 1}
Consider the linear Hopf flow (\ref{eq: curvature flow equation}) and (\ref{eq: quantised slope equ}) for $n\in (-1,1)$. Let $\mathcal{S}_0$ be a $C^4$-smooth strictly convex rotationally symmetric initial sphere. Assume that $\mathcal{S}_0$ satisfies the following conditions\\

\begin{enumerate}[label={\upshape(\arabic*)}]
    \item $\mathscr{L}^n_n\left(\frac{\displaystyle{s_0(\theta)}}{\displaystyle{\sin^{n+2}\theta}}\right)\in \lsin$,\\
    \item $n\cdot\lim\limits_{\theta \to 0} \left(\frac{\displaystyle{s_0(\theta)}}{\displaystyle{\sin^{2}\theta}}\right)=n\cdot\lim\limits_{\theta \to \pi} \left(\frac{\displaystyle{s_0(\theta)}}{\displaystyle{\sin^{2}\theta}}\right)=0$.
\end{enumerate}\vspace{0.3cm}

If the focal points of $\mathcal{S}_0$ at the north and south poles coincide, the flow converges to a round sphere of radius $\textstyle{-\frac{c}{a+b}}$.

Otherwise the flow converges to a non-round Hopf sphere with astigmatism at the equator given by the signed distance between the focal points at the poles of $S_0$.
\end{Thm}
\vspace{0.1in}

 If $\mathcal{S}_0$ has isolated umbilic points at the north or south pole, a sufficient condition to imply (2) can be given in terms of the surfaces \textit{umbilic slopes}, which quantify the rate at which $\mathcal{S}_0$ becomes umbilic (see Section \ref{ch: Rotationally Symmetric Surfaces} for details).

\begin{enumerate}[label={\upshape($2^*$)}]
\item At each pole of $\mathcal{S}_0$, the umbilic slope is greater than $3$.
\end{enumerate}

This is a corollary of Proposition \ref{prop: slope and order relation}. It is not immediately obvious which classes of surfaces satisfy condition (1). However a sufficient condition can be given in terms of the asymptotic fall-off of $s$ and its derivatives as $\theta \to 0,\pi$.
\begin{enumerate}[label={\upshape(3)}]
    \item $s_0^{(i)} \sim c_i\sin^{m-i} \theta$ at $\theta=0,\pi$ for $m>n+3$, $i=0,1,2$  with possibly different constants $c_i$ at $\theta=0$ and $\theta=\pi$.
\end{enumerate}
Condition (3) is in fact enough for convergence as it implies both conditions (1) and (2). Theorem \ref{thm: main thm 1} will be a consequence of the following more technical theorem.

\vspace{0.1in}
\begin{Thm}\label{thm: main thm 2}
 Let $S_0$ be a $C^2$-smooth strictly convex rotationally symmetric 2-sphere. If the astigmatism $s$ of $\mathcal{S}_0$ satisfies both of the following conditions\\
 
 \begin{enumerate}[label={\upshape(\Roman*)}]
    \item $\displaystyle{\frac{s}{\sin^{n+2}\theta} \in \lsin, \enskip \dv{s}{\theta}\in \text{\upshape{AC}}_{\text{\upshape{loc}}}(0,\pi), \enskip \mathscr{L}^n_n \left(\frac{s}{\sin^{n+2}\theta}\right)\in \lsin}$,\\
    \item $\displaystyle{\lim\limits_{\theta \to 0,\pi}\left( \sin^{2n+1}\theta \dv[]{}{\theta}\left(\frac{\displaystyle{s}}{\sin^{2n+2} \theta}\right)\right)=0}$,
\end{enumerate}\vspace{0.3cm}
then the conclusion of Theorem \ref{thm: main thm 1} holds.
\end{Thm}
\vspace{0.1in}

Theorem \ref{thm: main thm 2} has a weaker differentiability requirement of the initial surface than Theorem \ref{thm: main thm 1}. Indeed, (I) implies $\mathcal{S}_0$ is $C^3$-smooth everywhere apart from possibly at the north and south umbilic points, and $C^4$-smooth almost everywhere. Under the additional assumption that $\mathcal{S}_0$ is $C^4$-smooth, conditions (I) and (II) are implied by (1) and (2) respectively and Theorem \ref{thm: main thm 1} follows.

Our method of proof involves a spectral expansion in an eigenbasis of the non-integer Legendre operator, defined on the weighted Hilbert space $\lsin$. The eigenbasis is adapted to contain the stationary solutions of the linear Hopf flow. The restriction $n\in (-1,1)$, is needed to guarantee such an eigenbasis expansion is possible.

Section \ref{ch: Rotationally Symmetric Surfaces} fixes the notation used to describe geometrical quantities and derives some consequences of rotationally symmetry. Section \ref{ch: Linear Hopf Flow} discusses the geometrical properties of possible stationary solutions of the linear Hopf flow and describes the evolution of important geometric quantities under the flow. Critical to this description is the second order Legendre differential operator ${\mathscr L}_n^n$. 

Section \ref{ch: Singular Sturm-Liouville Operators} reviews the theory of singular Sturm-Liouville operators, which we require to prove the existence of the adapted eigenbasis. In particular the LC property of a singular Sturm-Liouville operator and boundary conditions required to generate possible self-adjoint domains are discussed. An application of the Spectral Theorem \ref{thm: spectral theorem} for LC operators then guarantees the existence of a complete orthonormal basis of eigenfunctions of such operators.

This is then applied in Section \ref{ch: Proof of The Main Theorem} where the main results are proven. This done by showing that ${\mathscr L}_n^n$ is LC for $n\in(-1,1)$ and by finding boundary conditions (namely conditions (I) and (II) of Theorem \ref{thm: main thm 2}) to define a self-adjoint domain for $\mathscr{L}^n_n$, which is appropriate for the flow.

The eigenbasis is then given explicitly in terms of Legendre functions and the various geometric quantities are similarly expressed. Finally the evolution is solved in terms of the eigenbasis as an expansion which decays exponentially in time to the stationary solution of the flow.

\vspace{0.1in}

\section{Rotationally Symmetric Surfaces} \label{ch: Rotationally Symmetric Surfaces}

The class of surfaces we concern ourselves within this work are elements of the set\vspace{0.3cm}
\[
\mathscr{W}=\left\{S\subset \mathbb{R}^3 : \parbox{23em}{$S$ is an embedded $C^2$-smooth topological 2-sphere which is rotationally symmetric and strictly convex.}\right\}.
\]
The terminology ``topological sphere'' which we will abbreviate as just ``sphere'' is taken to mean a closed surface of genus 0.
 
A given $\mathcal{S}\in \mathscr{W}$ will be orientated with outwards pointing normal $\hat{n}$ in a right handed coordinate system $\vec{X}=(x^1,x^2,x^3)$. Align the $x^3$-axis with the axis of rotational symmetry of $\mathcal{S}$. $\mathcal{S}$ is parametrized by pushing forward the standard polar coordinates $(\theta,\phi)$ of $S^2$ onto $\mathcal{S}$ by the inverse of the Guass map $\mathcal{N}^{-1}:S^2\to \mathcal{S}$. Hence $\theta\in[0,\pi]$ measures the angle made between the normal vector $\hat{n}$ of $\mathcal{S}$, and the axis of rotational symmetry whereas $\phi\in [0,2\pi]$ measures the angle made by a clockwise rotation from the $x^2x^3$-plane.

As a consequence of rotational symmetry many quantities $f:\mathcal{S}\to \mathbb{R}$ are independent of $\phi$, in which case we write for short-hand $f(\theta)$ instead of $f(\vec{X}(\theta,\phi))$. In particular the radii of curvature of $\mathcal{S}$, $r_1$ and $r_2$, associated to the meridian and profile principal foliations respectively, are functions of $\theta$ only. Of critical importance to our study is the astigmatism of $\mathcal{S}$ given by
\[
s=r_2-r_1.
\]
If $\mathcal{S}$ is assumed strictly convex then the radii are $C^{m}$-smooth functions of the parameter $\theta$ whenever $\mathcal{S}$ is $C^{m+2}$-smooth. Furthermore $s$ vanishes at, and only at, the umbilic points of $\mathcal{S}$. In particular $s(0)=s(\pi)=0$ and $s \equiv 0$ if and only if $\mathcal{S}$ is a round sphere. 

Our description of $\mathcal{S}$ in $\mathbb{R}^3$ will be facilitated by the support function $r:\mathcal{S}\to \mathbb{R}$ defined by $r=\vec{X}\cdot \hat{n}$. In the present setting $r$ is a function of $\theta$ only and $\vec{X}(\theta,\phi)$ can be recovered from $r$ and its derivatives via:
\begin{align*}
   &x^1+ix^2=(\sin \phi+i\cos \phi)\left(r\sin\theta  + \dv{r}{\theta}\cos\theta\right),
   &x^3=r\cos\theta -\dv{r}{\theta}\sin\theta.
\end{align*}
\begin{remark}\label{rem: focal point positions}
For later convenience, note that the focal points \cite{gk08} of $\mathcal{S}$ at $\theta=0$ and $\pi$ lie on the axis of rotational symmetry with $x^3$ coordinates 
\begin{align*}
    &f_0=r(0)-r_1(0), &f_\pi=-r(\pi)+r_1(\pi).
\end{align*}
This can be deduced by the above equation for $\vec{X}$ and the equations of the focal sheets $\vec{F}_i=\vec{X}-r_i \widehat{n}$, $i=1,2$.
\end{remark}
\vspace{0.1in}

We collect together some useful relationships between the above quantities on rotationally symmetric surfaces.
\begin{Prop} The following relationships hold
\begin{align}\label{eq: curvatures in terms of r}
    &r_1=\frac{\cos^2\theta}{\sin\theta}\dv[]{}{\theta}\left(\frac{r}{\cos\theta}\right), &r_2=\dv[2]{r}{\theta}+r,
\end{align}
\begin{equation}\label{eq: support function from quadratures}
r=C_2\cos\theta+C_1+\int \sin\theta \left[\int \frac{s}{\sin\theta}d \theta\right]d\theta,
\end{equation}
for constants $C_1$,$C_2$. If in addition surface is $C^3$, the derived Codazzi-Mainardi equation holds:
\begin{align}
    \label{eq: codazzi mainardi}
    \dv{r_1}{\theta}=(r_2-r_1)\cot\theta. 
\end{align}
\end{Prop}
\begin{proof}
The derivation of equation (\ref{eq: curvatures in terms of r}) can be found in \cite{gk10} or derived from the definition of the support function. Equations (\ref{eq: support function from quadratures}) and (\ref{eq: codazzi mainardi}) are derived by integrating or differentiating equation (\ref{eq: curvatures in terms of r}) respectively, with respect to $\theta$.
\end{proof}
\vspace{0.1in}

\begin{remark}\label{rem: par trans along normal lines}
Given $\mathcal{S}\in\mathscr{W}$ with support function $r$, the transformation $r\mapsto r+C_1$ translates $\mathcal{S}$ at each point along its normal line by a distance $C_1$, i.e $\mathcal{S}$ moves to a parallel surface. The transformation $r\mapsto r+C_2\cos\theta$ translates the entire surface a distance of $C_2$ along the $x^3$-axis. Therefore $s$ determines the oriented, affine normal lines of the surface. See \cite{gk04} for further details. 
\end{remark}
\vspace{0.1in}

\subsection{The Slope at an Isolated Umbilic}\hspace{0.1cm}\\

If $\mathcal{S}$ possesses isolated umbilic points at its north and south poles, i.e. at $\theta=0,\pi$, then we define the \textit{umbilic slopes} of $\mathcal{S}$ as
\begin{align*}
    &\mu_0=\lim_{\theta\to 0}\left(\frac{r_2(\theta)-r_2(0)}{r_1(\theta)-r_1(0)}\right), &\mu_\pi=\lim_{\theta\to \pi}\left(\frac{r_2(\theta)-r_2(\pi)}{r_1(\theta)-r_1(\pi)}\right).
\end{align*}
 If $\mathcal{S}$ is $C^3$-smooth, the umbilic slopes are just $\dv{r_2}{r_1}$ evaluated at $\theta=0$ and $\pi$ respectively.
 
 We remark that coordinate substitution $\theta \mapsto \pi-\theta$ corresponds to reversing the direction of the $x^3$ axis, transforming $\mu_0$ into $\mu_\pi$ and vice-versa. Hence arguments that are made concerning the umbilic at one pole will often hold at the other pole also. When this is the case we will simply say the argument follows by reflection.  \\
 
 The question of what values of umbilic slope are possible on various spheres is an area of active research \cite{Hopf1989a} \cite{Galvez22}. The derived Codazzi-Mainardi equation (\ref{eq: codazzi mainardi}) is a necessary integrability condition for a $C^3$-smooth surface to be rotationally symmetric and has some striking consequences in this direction, two of which are the following
 proposition and theorem.
\vspace{0.1in}

\begin{Prop}\label{prop: slope and order relation}
Let $\mathcal{S}\in \mathscr{W}$ be a $C^3$ sphere with an isolated umbilic at $\theta=0$. For any $\alpha \in \mathbb{R}$

\begin{enumerate}[label={\upshape(\arabic*)}]
    \item $\mu_0>\alpha+1 \implies \lim\limits_{\theta \to 0}\left(\frac{\displaystyle{s}}{\displaystyle{\sin^{\alpha}\theta}}\right)=0$,\\
    \item $\mu_0<\alpha+1 \implies \lim\limits_{\theta \to 0}\left(\frac{\displaystyle{s}}{\displaystyle{\sin^{\alpha}\theta}}\right)=\pm\infty$.
\end{enumerate}
If $\mathcal{S}$ has an isolated umbilic at $\theta=\pi$ then the value of $\mu_\pi$ dictates the behaviour of $s$ as $\theta\to\pi$ in the same way.
\end{Prop}
\begin{proof}
We prove the result for $\theta=0$, the $\theta=\pi$ case follows by reflection.  Since the umbilic at $\theta=0$ is isolated, there exists $\delta\in(0,\frac{\pi}{2})$ such that $s\neq0$ for $\theta\in(0,\delta]$. Therefore by the Codazzi-Mainardi equation (\ref{eq: codazzi mainardi}), $\dv{r_1}{\theta}\neq0$ on this interval also. Furthermore since the surface is assumed strictly convex and $C^3$-smooth, $\dv{r_2}{\theta}$ is continuous and bounded. Therefore $\dv{r_2}{r_1}$ is continuous on $(0,\delta]$. 

From the Codazzi-Mainardi equation and the definition of $s$ one can derive the separable ODE:
\[
\dv{r_2}{r_1}=1+\frac{\tan \theta}{s}\dv{s}{\theta}(\theta).
\]
Integrating from $\theta$ to $\delta$ gives
\begin{align*}
\left|s\right|&=|s(\delta)|\exp\left\{\int^\delta_\theta\cot\theta\left(1-\dv{r_2}{r_1}\right)d\theta\right\}\\
&=\left|\frac{s(\delta)}{\sin^\alpha\delta}\right|\sin^\alpha\theta\cdot \exp\left\{\int^\delta_\theta\cot\theta\left(\alpha+1-\dv{r_2}{r_1}\right)d\theta\right\},
\end{align*}
for all $\theta \in (0,\delta]$. Dividing both sides by $\sin^\alpha\theta$ we derive the relationship
\begin{align*}
    &\left|\frac{s(\theta)}{\sin^\alpha\theta}\right|=\left|\frac{s(\delta)}{\sin^\alpha\delta}\right|\exp\left\{\int^\delta_\theta\cot\theta\left(\alpha+1-\dv{r_2}{r_1}\right)d\theta\right\} &\forall \theta \in (0,\delta].
\end{align*}
Now let $\theta \to 0$. If $\mu_0>\alpha+1$ the quantity in parenthesis diverges to $-\infty$, implying $s/\sin^{\alpha}\theta \to 0$. On the other hand if $\mu_0<\alpha+1$ the quantity in parenthesis diverges to $+\infty$ which implies $|s/\sin^{\alpha}\theta| \to \infty$.
\end{proof}
\vspace{0.1in}

\begin{remark}
Proposition \ref{prop: slope and order relation} shows how the umbilic slopes dictate the rate of vanishing of $s$ at the north and south poles. We remark that if it is the case that $\mu_0$ or $\mu_\pi$ is equal to $\alpha+1$, then $s/\sin^{\alpha}\theta$ can exhibit either of the two behaviours in Proposition \ref{prop: slope and order relation} or tend to a non-zero constant, as illustrated by the three examples $s=\sin^\alpha\theta\cdot\ln\left(2 \csc \theta\right)^\epsilon$ for $\epsilon=-1,0,1$.
\end{remark}
\vspace{0.1in}

\begin{Thm}\label{thm: quantisation of smooth slope}
 If $\mathcal{S}\in \mathscr{W}$ is $C^\infty$ and has an isolated umbilic at the north or south pole, then the umbilic slope at that pole when it exists \& is finite, takes a value of an odd integer greater than or equal to $3$. 
 
If furthermore $\alpha+1$ is the value of the umbilic slope at a given pole, the limit of $s/\sin^{\alpha}\theta$ as we approach the pole is finite, non-zero.
 \end{Thm}
 \begin{proof}
Let $\mu$ be the umbilic slope at $\theta=0$ and assume it is finite. Since $\mathcal{S}$ is strictly convex, rotationally symmetric and smooth, $s$ and all odd derivatives of $s$ vanish at $\theta=0$, i.e. $s(0)=s^{(2m+1)}(0)=0$ for all $m\in\mathbb{N}$.
 
 Now assume for contradiction that all even derivatives also vanish. Then $s^{(m)}(0)=0$ for all $m\in\mathbb{N}_0$. In particular for any $\beta \in \mathbb{N}$, by L'H\^opitals rule
 \[
 \lim_{\theta \to 0}\left(\frac{s}{\sin^\beta \theta}\right)=\ldots=\frac{1}{\beta !}s^{(\beta)}(0)=0.
 \]
 Using the contrapositive of (2) in Proposition \ref{prop: slope and order relation}, it follows that $\mu\geq \beta+1$ which contradicts the assumption of $\mu$ being finite. Therefore there exists some $k\in\mathbb{N}$ such that $s^{(2k)}(0)\neq0$, without loss of generality take $k$ to be the smallest natural number such that this holds, so $s^{(m)}(0)=0$ for all $m<2k$. The smoothness assumption on $\mathcal{S}$ grantees the existence of $s^{(2k)}(0)$ and therefore we have the following limit
 
 \[
  \lim_{\theta \to 0}\left(\frac{s}{\sin^{2k} \theta}\right)=\ldots=\frac{1}{(2k)!}s^{(2k)}(0)\neq0,\pm \infty.
 \]
 
 This time using the contrapositive of both (1) and (2) in Proposition \ref{prop: slope and order relation} we have $\mu \geq 2k+1$ and $\mu \leq 2k+1$ which proves the first claim. The second claim follows from the above limit. If the isolated umbilic is at $\theta=\pi$ we argue by reflection.
 \end{proof}
 
 \vspace{0.1in}

Although the umbilic slopes are not generally quantised for non-smooth spheres with isolated umbilic points, the regularity of a sphere still places restrictions on the possible values of the umbilic slope.

\vspace{0.1in}

\begin{Prop}\label{prop: If C^4 then surface has umbilic slope gre 3 }
If $\mathcal{S}\in\mathscr{W}$ is of regularity $C^4$ and has isolated umbilic points, then the umbilic slopes of $\mathcal{S}$ are greater or equal to $3$. 
\end{Prop}
\begin{proof}
First argue at the $\theta=0$ umbilic. By L'H\^opital and the assumed regularity of $\mathcal{S}$, we have the existence of the following limit:
\begin{equation}\label{eq: limit of s/sin from l'hopital}
\lim_{\theta \to 0}\left(\frac{s}{\sin^2\theta}\right)=\lim_{\theta \to 0}\left(\frac{\dv{s}{\theta}}{2\cos\theta\sin\theta}\right)=\left.\frac{1}{2}\dv[2]{s}{\theta}\right|_{\theta=0}
\end{equation}
Hence by the converse of (2) in Proposition (\ref{prop: slope and order relation}), we have that $\mu_0\geq 3$. The case at $\theta=\pi$ follows by reflection.
\end{proof}
\section{Linear Hopf Flow} \label{ch: Linear Hopf Flow}

\subsection{Stationary Solutions}
The stationary solutions of the linear Hopf flow (\ref{eq: curvature flow equation}) are surfaces satisfying the curvature relationship
\begin{equation}\label{eq: linear curvature relationship}
    0=ar_1+br_2+c.
\end{equation}
 The surfaces in $\mathscr{W}$ for which a general linear curvature relationship such as equation (\ref{eq: linear curvature relationship}) holds are called \textit{linear Hopf spheres}. 
 
 \vspace{0.1in}
 
 \begin{Prop}\label{prop: Hopf sphere geometric quantities}
 The linear Hopf spheres which solve equation (\ref{eq: linear curvature relationship}) for given parameter values $a$, $b$ and $c$ have astigmatisms of the form
\begin{equation}\label{eq: astigmatism of hopf sphere}
s_\text{\tiny{Hopf}}=C_0\sin^{2n+2}\theta,
\end{equation}
and support functions of the form
\begin{equation}\label{eq: support function of hopf sphere}
r_\text{\tiny{Hopf}}=C_2\cos\theta + C_1+C_0\left[\frac{\sin^{2n+2}\theta}{2n+2}-\cos\theta\int^\theta_0\sin^{2n+1}\theta d\theta\right],
\end{equation}
where $C_0$,$C_1$ and $C_2$ are constants and $-a/b=2n+3$.
 \end{Prop}
\begin{proof}
Combing the derived Codazzi-Mainardi equation (\ref{eq: codazzi mainardi}) with the curvature relationship (\ref{eq: linear curvature relationship}) results in a separable ODE which is solved to give equation (\ref{eq: astigmatism of hopf sphere}). The support function (\ref{eq: support function of hopf sphere}) follows by integrating equation (\ref{eq: astigmatism of hopf sphere}) by quadrature as in equation (\ref{eq: support function from quadratures}).
\end{proof}
\vspace{0.1in}

Any surface with isolated umbilic points satisfying equation (\ref{eq: linear curvature relationship}) necessarily has umbilic slopes taking the common value
\[
\mu_0=\mu_\pi=-a/b.
\]
Furthermore, if this surface is in $\mathscr{W}$ and smooth, Theorem \ref{thm: quantisation of smooth slope} implies $-a/b=2n+3$ for some $n\in \mathbb{N}_0$. Therefore the flow can only have smooth, non-round stationary solutions for such values of $-a/b$, i.e. the only smooth linear Hopf spheres are ones with odd umbilic slope.  

For this reason the work undertaken in \cite{gk21} investigated the linear Hopf flow with $-a/b$ restricted to an odd integer $\geq 3$, on smooth initial surfaces in $\mathscr{W}$. The central result is then:

\vspace{0.1in}
\begin{Thm} \cite{gk21} \label{thm: Integer Lin Hopf flow theorem}
Let $\mathcal{S}_{0}\in\mathscr{W}$ be a smooth initial surface with equal umbilic slopes $\mu=\mu_0=\mu_\pi$. The linear Hopf flow (\ref{eq: curvature flow equation}) and (\ref{eq: quantised slope equ}) behaves in the following manner:
\begin{enumerate}[label={\upshape(\arabic*)}]
    \item if $2n+3<\mu$, the evolving sphere converges exponentially through smooth convex spheres to the round sphere of radius $\frac{c}{2(n+1)}$,\\
    \item if $2n+3=\mu$, an initial non-round sphere converges exponentially thorough smooth convex spheres to a non-round linear Hopf sphere,\\
    \item if $2n+3>\mu$, the sphere diverges exponentially.
\end{enumerate}
\end{Thm}
\vspace{0.1in}

Since in the current paper we consider $n$ which is generically non-integer, we cannot expect generically smooth, non-round stationary solutions as in Theorem \ref{thm: Integer Lin Hopf flow theorem}.

We also remark that the flow in Theorem \ref{thm: Integer Lin Hopf flow theorem} fixes $b=1$, where as in this work we allow $b>0$. These conventions are equivalent up to a parabolic scaling $t \mapsto bt$ of equation (\ref{eq: curvature flow equation}).
\vspace{0.1in}

\subsection{Time Evolution of Geometric Quantities}

The curvature flow equation (\ref{eq: curvature flow equation}) is equivalent to the following evolution equation for the support function
\begin{align}\label{eq: evolution of support function}
    \pdv[]{r}{t}=b\frac{\partial^2r}{\partial\theta^2}+a\cot\theta \frac{\partial r}{\partial\theta}+(a+b)r+c,
\end{align}
as can be seen from the definition of $r$ and equations (\ref{eq: curvatures in terms of r}). We remark that the coordinate singularities in equation (\ref{eq: evolution of support function}) prevent us from using regular Sturm-Liouville theory to derive an associated eigenbasis, which motivates the singular theory discussed in Section \ref{ch: Singular Sturm-Liouville Operators}.

The solution $r(t,\theta)$ to equation (\ref{eq: evolution of support function}) may be found by first considering the behaviour of the astigmatism $s(t,\theta)$. The support function $r(t,\theta)$ may then be recovered via quadrature by equation (\ref{eq: support function from quadratures}) up to two time-dependent constants determined by an initial condition and equation (\ref{eq: evolution of support function}). 

\begin{Prop}
Under the linear Hopf flow the astigmatism evolves as
\begin{equation}\label{eq: s/sin evolution}
\pdv{}{t}\left(\frac{s}{\sin^{n+2}\theta}\right)=b\cdot \mathscr{L}^n_n\left(\frac{s}{\sin^{n+2}\theta}\right),
\end{equation}
where $\theta \in (0,\pi)$,  $-a/b=2n+3$ and
\begin{equation}\label{e:legop}
\mathscr{L}^\mu_\nu=\pdv[2]{}{\theta}+\cot\theta\pdv{}{\theta}
+(\nu+1)\nu-\frac{\mu^2}{\sin^2\theta},
\end{equation}
is the Legendre operator.
\end{Prop}
\begin{proof}
Differentiating equation (\ref{eq: support function from quadratures}) twice gives a relationship between $s$ and the derivatives of $r$.  The claim follows after inserting this relationship into equation (\ref{eq: evolution of support function}) and performing some algebraic manipulation.
\end{proof}
\begin{remark}\label{rem: hopf spheres are stationary}
As expected the astigmatism of the appropriate linear Hopf sphere is stationary under the flow since 
\[
\frac{s_\text{\tiny{Hopf}}}{\sin^{n+2}\theta}=\sin^n\theta\in \text{Ker }\mathscr{L}^n_n.
\]
\end{remark}
\vspace{0.1in}

A crucial ingredient for solving equation (\ref{eq: s/sin evolution}) in the case $n\in\mathbb{N}$, is that the eigenfunctions of $\mathscr{L}^n_n$ are the associated Legendre polynomials. The associated Legendre polynomials, $\tp^n_m(\cos\theta)$, are polynomials in sine and cosine and form an orthogonal basis of $C^0[0,\pi]$. In particular they span the higher index terms ($l\geq n$) of the following astigmatism decomposition for smooth surfaces:
\begin{align}\label{eq: integer astigmatism decomposition}
    s&=\sum\limits^\infty_{l=0}(a_l+b_l\cos \theta)\sin^{2l+2}\theta.
\end{align}

 In \cite{gk21} this decomposition enabled the flow to be solved analytically when $n\in\mathbb{N}$ by projecting the flow equation (\ref{eq: s/sin evolution}) into each eigenbasis of $\mathscr{L}^n_n$ and solving for the time dependency of the coefficients $a_l(t)$,$b_l(t)$. 
 
 The term with coefficient $a_n$ is the astigmatism of the linear Hopf sphere with umbilic slope $2n+3$, therefore in the case of convergence of $s$ to a linear Hopf sphere, $b_l \to 0$ for all $l \in \mathbb{N}_0$ and $a_l \to 0$ for all $l \in \mathbb{N}_0\backslash\{n\}$ as $t \to \infty$, leaving only the linear Hopf term. In the case of convergence to a round sphere, $a_n=0$.
 
 We try to emulate the above argument in this paper for $n\notin \mathbb{N}$. In the non-integer case, the eigenfunctions of $\mathscr{L}^n_n$ are no longer trigonometric polynomials and are not terms in the series expansion (\ref{eq: integer astigmatism decomposition}). Furthermore it is no longer clear if the non-integer Legendre functions are orthogonal or form a basis. To address this we will write the Legendre operator in its Sturm-Liouville form
 \begin{equation}\label{eq: Legendre Operator}
     \mathscr{L}^\mu_\nu=\frac{1}{\sin\theta}\left[\frac{\partial}{\partial\theta}\left(\sin\theta\frac{\partial}{\partial\theta}\right)+\nu(\nu+1)\sin\theta-\frac{\mu^2}{\sin\theta}\right]
 \end{equation}
 and show using Singular Sturm-Liouville theory that we can find an orthogonal basis in which the surfaces astigmatism can be decomposed.
\vspace{0.1in}

\section{Singular Sturm-Liouville Operators} \label{ch: Singular Sturm-Liouville Operators}

In this section we summarize the theory of singular Sturm-Liouville problems following \cite{zettl}. In Section \ref{ch: Proof of The Main Theorem} the theory will be used to deduce the existence of an orthogonal eigenbasis for the non-integer Legendre operator associated with the linear Hopf flow.

Consider the general Sturm-Liouville operator
\begin{equation}\label{eq: General SL operator}
    T=-\frac{1}{w(x)}\left[\dv{}{x}\left(p(x) \dv{}{x}\right)+q(x)\right],
\end{equation}
where $1/p,q,w \in L_{loc}\big((a,b);\mathbb{R}\big)$ are locally Lebesgue integrable real-valued functions on the interval $(a,b)$ and $w>0$. Operators of this form are called \textit{singular Sturm-Liouville operators}. The Legendre operator discussed in Section \ref{ch: Linear Hopf Flow} is an example of such. 

View $T$ as a linear operator on $L^2\big((a,b),w(x)dx;\mathbb{C}\big)$, the Hilbert space of complex-valued square integrable functions with weight $w$, denoted simply by $L^2_{w}(a,b)$. Under the standard inner-product of the $\lw$ spaces
\[
\langle f,g\rangle_w=\int^b_a f(x)\cdot\overline{g(x)}\cdot w(x) dx,
\]
the operator $T$ satisfies the so-called Greens formula
\begin{equation}\label{eq: Greens formula}
    \langle Tf,g\rangle_w=p(x)\left(f(x)\overline{g'(x)}-f'(x)\overline{g(x)}\right)\bigg|^{x=b^-}_{x=a^+}+ \langle f,Tg\rangle_w,
\end{equation}
where the evaluation of the boundary term is to be understood as a limit. Greens formula allows us to investigate the symmetry of $T$ in $\lw$ so long as the functions $f$ and $g$ are chosen to be such that the terms in (\ref{eq: Greens formula}) are well defined. For this purpose the \textit{maximal domain} is introduced:
\begin{equation*}
    \dmax=\big \{f \in \lw : f,pf'\in \text{AC}_{\text{loc}}(a,b), \enskip Tf\in \lw\big \},
\end{equation*}
 where $\text{AC}_{\text{loc}}(a,b)$ is the space of functions which are absolutely continuous on all compact intervals of $(a,b)$. 
 
 The requirement that $f,pf'\in \text{AC}_{\text{loc}}(a,b)$, is enough to ensure that $f$, $pf'$ are differentiable almost everywhere and their derivatives are Lebesgue integrable, which gives meaning to equation (\ref{eq: Greens formula}). 
 
 In the case that the coefficient $p$ satisfies $1/p \in \text{AC}_{\text{loc}}(a,b)$, the description of $\dmax$ simplifies to
\begin{equation}\label{eq: simp dmax}
    \dmax=\big \{f \in \lw : f'\in \text{AC}_{\text{loc}}(a,b), \enskip Tf\in \lw\big \},
\end{equation}
in particular the elements of $\dmax$ must be $C^1(a,b)$ and have second derivative a.e. Such is the case with the Legendre operator (\ref{eq: Legendre Operator}).

Sturm-Liouville operators often come supplied with boundary conditions as to make the boundary term in equation (\ref{eq: Greens formula}) vanish, i.e.
\begin{align}
    \langle Tf,g\rangle_w=\langle f,Tg\rangle_w.
\end{align}
These boundary conditions constitute part of $T$'s domain of definition, which naturally must be a subset of $\dmax$. 

In non-singular Sturm-Liouville theory\footnote{which requires the stronger condition $1/p,q,w$ are Lebesgue integrable over $(a,b)$} the well-known boundary conditions
\begin{align}\label{eq: regular sl boundary conditions}
&\alpha_1f(a)+\alpha_2p(a)f'(a)=0, &\beta_1f(b)+\beta_2p(b)f'(b)=0, && \alpha_1,\alpha_2,\beta_1,\beta_2 \in \mathbb{C},
\end{align}
define a domain $\dsa$ for $T$ which make $T|_{\dsa}$ self adjoint \cite{naimark68}. It is for this reason $\dsa$ will be referred to as a self adjoint domain for $T$.

In singular problems however, the quantities $f(x)$ and $p(x)f'(x)$ may not exist as $x\to a$ or $b$, even if $f\in \dmax$. Therefore boundary conditions such as (\ref{eq: regular sl boundary conditions}) are not appropriate.  To facilitate the description of boundary conditions for singular problems the \textit{Lagrange bracket}
\begin{equation}\label{eq: lagrange bracket}
    [f,g]_p(x)=p(x)\left(f(x)\overline{g'(x)}-f'(x)\overline{g(x)}\right),
\end{equation}
is introduced. Unlike the terms in boundary condition (\ref{eq: regular sl boundary conditions}), the Lagrange bracket $[f,g]_p(x)$ \textit{is} finite in the limits $x\to a,b$ so long as $f,g\in\dmax$. 

In order to give appropriate boundary conditions in the singular case, we first give a definition.
\vspace{0.1in}

\begin{Def}\label{def: LC/LP}
Given a singular Sturm-Liouville operator $T$, we say $T$ is \textit{limit-circle} (LC) at $x=a$ if for a given $\chi \in \mathbb{C}$ all solutions of the eigenvalue equation
\[
Ty=\chi y,
\]
are in $L^2_w(a,c)$ for some $c\in(a,b)$. Otherwise we say $T$ is \textit{limit-point} (LP) at $a$.

Similarly we say $T$ is LC at $x=b$ if correspondingly $y\in L^2_w(c,b)$ and LP at $b$ otherwise.

$T$ is said to be LC(LP) if it is LC(LP) at both $a$ and $b$.
\end{Def}
\vspace{0.1in}

\begin{remark}
$T$ being LC or LP is independent of $\chi \in \mathbb{C}$ \cite{zettl}.
\end{remark}
\vspace{0.1in}

The next theorem states the parallel of boundary condition (\ref{eq: regular sl boundary conditions}) for LC Sturm-Liouville operators.
\vspace{0.1in}

\begin{Thm}\cite{zettl}\label{thm: Zettle bc}
Let $T$ be a LC Sturm-Liouville operator and $\eta,\psi$ be real valued functions in $D_{\text{max}}$ such that $[\eta,\psi]_p(a)=1$ and $[\eta,\psi]_p(b)=1$. Consider the separated boundary condition
\begin{align}\label{eq: genral seperated S.A. B.C. at a }
    &\alpha_1 [u,\eta]_p(a)-\alpha_2 [u,\psi]_p(a)=0 &(\alpha_1,\alpha_2)\in\mathbb{R}^2\backslash \{0\}\\
    \label{eq: genral seperated S.A. B.C. at b }
    &\beta_1 [u,\eta]_p(b)-\beta_2 [u,\psi]_p(b)=0 &(\beta_1,\beta_2)\in\mathbb{R}^2\backslash \{0\},
\end{align}
where
\begin{align*}
&[u,f]_p(a)=\lim_{x\to a^+}[u,f]_p(x) & [u,f]_p(b)=\lim_{x\to b^-}[u,f]_p(x).
\end{align*}
Then given $\alpha_i$, $\beta_i$, $i=1,2$ as above, the domain 
\[
\dsa=\{y\in D_{\text{max}} : \text{equations(\ref{eq: genral seperated S.A. B.C. at a })-(\ref{eq: genral seperated S.A. B.C. at b }) hold}\},
\]
is a self adjoint-domain for $T$, i.e. $T|_{\dsa}$ is a self-adjoint operator.
\end{Thm}
\vspace{0.1in}

In addition
\vspace{0.1in}
\begin{Prop}\label{prop: compact res, simple eigval}\cite{Teschl}
If $T$ is a LC Sturm-Liouville operator then $T$ has a compact resolvent.
\end{Prop}
\vspace{0.1in}

Together Theorem \ref{thm: Zettle bc} and Proposition \ref{prop: compact res, simple eigval} are enough to give us an orthonormal basis due to the following spectral theorem:

\vspace{0.1in}
\begin{Thm}[Spectral Theorem]\cite{Taylor}\label{thm: spectral theorem}
Let $\mathcal{H}$ be a separable complex Hilbert space with inner product $\langle\cdot,\cdot\rangle$ and let $T:\dom{T}\subseteq \mathcal{H} \to \mathcal{H}$ be a linear, self-adjoint operator on $\mathcal{H}$ with compact resolvent. Then, there exists a sequence $(\lambda_n)_n \subset \mathbb{R}$ and a complete orthonormal basis $(e_n)_n$ of $\mathcal{H}$ with $e_n \in \dom{T}$ for all $n\in \mathbb{N}$ such that

\begin{enumerate}[label={\upshape(\arabic*)}]
    \item $Te_n=\lambda_ne_n$,\\
    \item $\dom{T}=\{x\in\mathcal{H} | (\lambda_n\langle x, e_n \rangle )_n \in \ell^2\}$,\\
    \item $Tx=\sum\limits_{n=1}^\infty \lambda_n\langle x, e_n \rangle e_n$ for all $x \in \dom{T}.$
\end{enumerate}
\end{Thm}
\vspace{0.1in}

\section{Proof of Theorem \ref{thm: main thm 2}} \label{ch: Proof of The Main Theorem}

We will solve equation (\ref{eq: s/sin evolution}) in terms of an eigenbasis expansion of the operator $\mathscr{L}^n_n$. Theorem \ref{thm: main thm 2} will then follow by the asymptotic behaviour of the solution as $t \to \infty$. 

The proof is organised into three parts: Firstly we show the existence of the appropriate eigenbasis using the theory in Section \ref{ch: Singular Sturm-Liouville Operators}. Secondly we determine the basis explicitly as the Legendre functions $\{\tp^{-n}_{n+m}(\cos\theta)\}^\infty_{m=0}$ and derive the corresponding expansions for $s$, $r_1$ and $r$. Finally, we use this basis to solve the time evolution problem
\begin{align}\label{eq: time evolution problem in lsin}
    \begin{cases}
    \partial_t u=b \cdot \mathscr{L}^n_n u, & [0,\infty) \times [0,\pi]\\
    u(t,\cdot)\in \dsa, & t\in [0,\infty)\\
    u=u_0, & \{t=0\} \times [0,\pi].
    \end{cases}
\end{align}
The function space $\dsa$ is a self adjoint domain for $\mathscr{L}^n_n$ and plays the role of an effective a parabolic boundary condition in (\ref{eq: time evolution problem in lsin}). If $s$ is the astigmatism of a rotationally symmetric surface, making the substitution $u=s/\sin^{n+2}\theta$ turns the time evolution problem (\ref{eq: time evolution problem in lsin}) into one describing the evolution of a surfaces astigmatism under the linear Hopf flow, i.e. equation (\ref{eq: s/sin evolution}). The solution $s(t,\theta)$ is then integrated for the support function $r(t,\theta)$ and both are shown to exhibit the asymptotic behaviour as $t \to \infty$
\begin{align}\label{eq: asymptotic behaviour of geoemtric quant.}
 &s(t,\theta)\sim \widetilde{\gamma}\cdot s_\text{\tiny{Hopf}}, &r(t,\theta)\sim r_\text{\tiny{Sphere}}+\widetilde{\gamma}\cdot r_\text{\tiny{Hopf}},
\end{align}
where $r_\text{\tiny{Sphere}}$ is the support function of a sphere with radius $-\frac{c}{a+b}$ and $\widetilde{\gamma}$ is the signed distance between the focal points of the initial surface at $\theta=0$ and $\pi$.

\subsection{Existence of the Eigenbasis and The Self-Adjoint Domain}\hspace{1cm}\\

First we remark for which values of $n$ that $\mathscr{L}^n_n$ is LC.
\begin{Prop}\label{prop: L^n_n is LC}
$\mathscr{L}^n_n$ is LC if and only if $n\in(-1,1)$.
\end{Prop}
\begin{proof}
To check if $\mathscr{L}^n_n$ is LC, solve the eigenvalue problem $\mathscr{L}^n_nu=\chi u$. Recall when checking LC/LP, we are free to choose $\chi$ as we please. If we set $\chi:=-n(n+1)$ we must solve the problem
$$
\frac{1}{\sin \theta}\dv{}{\theta}\left(\sin \theta \dv{u}{\theta}\right)-\frac{n^2}{\sin^2\theta}u=0.
$$
If we can show that any two linearly independent solutions are square integrable (with weight $w=\sin \theta$), it follows that every solution is square integrable by the triangle inequality.
First assume $n=0$. Then the eigenvalue problem is simply
\[
\frac{1}{\sin \theta}\dv{}{\theta}\left(\sin \theta \dv{u}{\theta}\right)=0,
\]
with two linearly independent solutions $\ln \cot \left(\frac{\theta}{2}\right)$ and a constant function. These are square integrable. Now assume $n\neq0$, then two linearly independent solutions are
$$
u_\pm=\cot^{\pm n}\left(\frac{\theta}{2}\right).
$$
We have then,
$$||u_\pm||^2_{\lsin}=\int^{\pi}_0 \cot^{\pm 2n}\left(\frac{\theta}{2} \right) \cdot \sin \theta d\theta=2\int^{\pi}_0 \left[\cos\left(\frac{\theta}{2}\right) \right]^{1\pm 2n}\cdot \left[\sin\left(\frac{\theta}{2}\right) \right]^{1\mp 2n} d\theta, $$
which is convergent if and only if $-1<n<1$ and therefore $\mathscr{L}^n_n$ is LC if and only if $n\in(-1,1)$.
\end{proof}
\vspace{0.1in}

For such values of $n$ we may now use Theorem \ref{thm: Zettle bc} to find self-adjoint domains for $\mathscr{L}^n_n$. Furthermore, as to make the convergence obvious,  we'd like the eigenbasis of $\mathscr{L}^n_n$ to contain explicitly the stationary solution to equation (\ref{eq: s/sin evolution}):
\[
\frac{s_\text{\tiny{Hopf}}}{\sin^{n+2}\theta},
\]
it is therefore necessary that $\frac{s_\text{\tiny{Hopf}}}{\sin^{n+2}\theta} \in \dsa$.
Finding such a self-adjoint domain will be the content of the next proposition.

\vspace{0.1in}

\begin{Prop}\label{prop: Self-Adjoint Domain for L^n_n}
The separated boundary condition\\
\begin{align}\label{eq: seperated boundary conditions for dsa.}
    &\lim\limits_{\theta \to 0}\left[ \sin^{2n+1}\theta \dv[]{}{\theta}\left(\frac{u}{\sin^n \theta}\right)\right]=0, & \lim\limits_{\theta \to \pi}\left[ \sin^{2n+1}\theta \dv[]{}{\theta}\left(\frac{u}{\sin^n \theta}\right)\right]=0,
\end{align} \hspace{0.1cm}\\
generates a self-adjoint domain for the Legendre operator\\
\begin{equation}
D_{\text{S.A.}}=D_{\text{max}}\cap \left\{u \in \lsin: \text{ boundary condition }(\ref{eq: seperated boundary conditions for dsa.})\text{ holds.} \right\},
\end{equation}
where $\dmax$ is given by equation (\ref{eq: simp dmax}) for $(a,b)=(0,\pi)$ and $T=\mathscr{L}^n_n$. Furthermore $s_\text{\tiny{Hopf}}/\sin^{n+2}\theta \in \dsa$.
\end{Prop}
\begin{proof}
The expression for $\dmax$ for the Legendre operator takes the form
\begin{equation*}\label{eq: dmax for L^n_n}
    \dmax=\big \{u \in \lsin : {\Large \textstyle \dv{u}{\theta}}\in \text{AC}_{\text{loc}}(0,\pi), \enskip \mathscr{L}^n_nu\in \lsin\big \}.
\end{equation*}
Theorem \ref{thm: Zettle bc} tells us that all self adjoint domains of the Legendre operator generated by separated boundary conditions are given by functions $u\in \dmax$ satisfying 
\[
\alpha_1[u,\eta]_{\sin\theta}(0)+\alpha_2[u,\psi]_{\sin\theta}(\pi)=0 \enskip \& \enskip \beta_1[u,\eta]_{\sin\theta}(0)+\beta_2[u,\psi]_{\sin\theta}(\pi)=0,
\]
 where $\left\{\eta,\psi\right\} \subset D_{\text{max}}$ and $[\eta,\psi]_{\sin\theta}(0)=[\eta,\psi]_{\sin\theta}(\pi)=1$, with $(\alpha_1,\alpha_2),(\beta_1,\beta_2)\in\mathbb{C}^2 \backslash\{0\}$.
 
Take
$(\eta,\psi)=\frac{1}{\sqrt{n}}\left(\sin^n\theta,\sin^{-n}\theta\right)$.
It is straightforward to check that this choice of $\eta$ and $\psi$ satisfy the above requirements. From Theorem \ref{thm: Zettle bc}, the boundary conditions with this choice of $\eta$ and $\psi$ are
\begin{align}
   \label{eq: full seperated bc 1}
   &\lim\limits_{\theta \to 0}\left[ \alpha_1 \sin^{2n+1}\theta \dv[]{}{\theta}\left(\frac{u}{\sin^n \theta}\right)+\frac{\alpha_2}{ \sin^{2n-1}\theta} \dv[]{}{\theta}\left(u\cdot \sin^n \theta\right)\right]=0,\\
   \label{eq: full seperated bc 2}
   &\lim\limits_{\theta \to \pi}\left[ \beta_1 \sin^{2n+1}\theta \dv[]{}{\theta}\left(\frac{u}{\sin^n \theta}\right)+\frac{\beta_2}{ \sin^{2n-1}\theta} \dv[]{}{\theta}\left(u\cdot \sin^n \theta\right)\right]=0.
\end{align}
If we let $u_\text{\tiny{Hopf}}=s_\text{\tiny{Hopf}}/\sin^{n+2}\theta$ then
\[
\sin^{2n+1}\theta \dv[]{}{\theta}\left(\frac{u_\text{\tiny{Hopf}}}{\sin^n \theta}\right)\equiv 0, \qquad \frac{1}{ \sin^{2n-1}\theta} \dv[]{}{\theta}\left(u_\text{\tiny{Hopf}}\cdot \sin^n \theta\right)=2nC\cos\theta,
\]
and so boundary conditions (\ref{eq: full seperated bc 1})-(\ref{eq: full seperated bc 2}) are satisfied by $u_\text{\tiny{Hopf}}$ only when $\alpha_2=\beta_2=0$. Therefore setting $\alpha_1=\beta_1=1$ and $\alpha_2=\beta_2=0$, gives the self adjoint domain
\[
\dsa=D_{\text{max}}\cap \left\{u \in \lsin: \lim\limits_{\theta \to 0,\pi}\left[ \sin^{2n+1}\theta \dv[]{}{\theta}\left(\frac{u}{\sin^n \theta}\right)\right]=0 \right\}.
\]
It is easily seen that $u_\text{\tiny{Hopf}} \in \dmax$ and so we are done.
\end{proof}

\vspace{0.1in}

Given the astigmatism of a surface, we may characterise when $s/\sin^{n+2}\theta \in \dsa$ in terms of the astigmatism.
\vspace{0.1in}

\begin{Prop}
If $s$ is the astigmatism of $\mathcal{S}\in\mathscr{W}$, then $s/\sin^{n+2}\theta \in \dsa$ if and only if {\upshape{(I)}} and {\upshape{(II)}} of Theorem \ref{thm: main thm 2} hold.
\end{Prop}
\begin{proof}
We make the substitution $u=s/\sin^{n+2}\theta$, the maximal domain conditions (\ref{eq: dmax for L^n_n}) and boundary condition (\ref{eq: seperated boundary conditions for dsa.}) become (I) and (II) respectively.
\end{proof}
\vspace{0.1in}

Under stronger assumptions on the surfaces regularity, conditions (I) and (II) can be replaced by the following sufficient conditions:

\begin{Prop}\label{lem: s.a. boundary conditions in terms of s, when C4}
If $\mathcal{S}\in\mathscr{W}$ is in addition $C^4$-smooth, it is sufficient for $s/\sin^{n+2}\theta \in \dsa$, that both {\upshape(1)} and {\upshape(2)} of Theorem \ref{thm: main thm 1} hold.
\end{Prop}
\begin{proof}
We will prove that under the assumption $\mathcal{S}$ is $C^4$, (1) implies (I) and (2) implies (II). Firstly since $\mathcal{S}\in\mathscr{W}$ and is $C^4$, we have $\dv{s}{\theta}\in \text{AC}_{\text{loc}}(0,\pi)$. Also, by arguing as in Proposition (\ref{prop: If C^4 then surface has umbilic slope gre 3 }), $s$ has the following asymptotic behaviours near the boundary of $(0,\pi)$
\[
\frac{s}{\sin^{n+2}\theta} \sim \frac{1}{2\sin^n \theta}\left. \dv[2]{s}{\theta}\right|_{\theta=0} \text{ as } \theta \to 0 \quad\text{ and }\quad \frac{s}{\sin^{n+2}\theta} \sim \frac{1}{2\sin^n \theta}\left. \dv[2]{s}{\theta}\right|_{\theta=\pi} \text{ as } \theta \to \pi.
\]
We state the following easy Lemma:
\begin{Lem}\label{lem: L2 sufficent boundary growth}
A continuous function $f$ on $(0,\pi)$ satisfying the growth conditions
\[
f \sim \frac{k_0}{\sin^p \theta} \text{ as } \theta \to 0 \quad \text{ and }\quad f \sim \frac{k_\pi}{\sin^p \theta} \text{ as } \theta \to \pi,  \qquad p<1,
\]
for constants $k_0$ and $k_\pi$ is square integrable with weight $\sin\theta$, i.e. $f\in \lsin$.
\end{Lem}

Therefore, since $n<1$, lemma \ref{lem: L2 sufficent boundary growth} gives $s/\sin^{n+2}\theta \in \lsin$. Hence the only requirement left for $s$ to satisfy (I) is that $\mathscr{L}^n_n\left(s/\sin^{n+2}\theta\right)\in \lsin$ which is (1).

Secondly, we re-write condition (II) as
\[
\lim\limits_{\theta \to 0,\pi}\left(\frac{\dv{s}{\theta}}{\sin \theta}-2(n+1)\cos\theta \frac{s}{\sin^2\theta}\right)=0.
\]
Substituting the limits in Proposition \ref{eq: limit of s/sin from l'hopital} in to the above gives condition (II) as the pair of equations
\begin{align*}
    &n \cdot \lim_{\theta \to 0}\left(\frac{s}{\sin^2\theta}\right)=0, & n \cdot \lim_{\theta \to \pi}\left(\frac{s}{\sin^2\theta}\right)=0,
\end{align*}
which is (2).
\end{proof}
\vspace{0.1in}

\begin{Prop}
Let $\mathcal{S}\in\mathscr{W}$ be $C^4$ and satisfy {\upshape(3)}. Then both {\upshape{(1)}} and {\upshape{(2)}} of Theorem \ref{thm: main thm 1} hold.
\end{Prop}
\begin{proof}
First we show that (3) implies (1). By lemma \ref{lem: L2 sufficent boundary growth}, it is enough to show that (3) implies
\[
\mathscr{L}^n_n \left(\frac{s}{\sin^{n+2}\theta}\right) \sim \frac{k}{\sin^p \theta},
\]
around $\theta=0$ for some $p<1$ and some $k$, and likewise at $\theta=\pi$. We have

\begin{align*}
\mathscr{L}^n_n\left(\frac{s}{\sin^{n+2}\theta}\right)&=\frac{1}{\sin^{n+2}\theta}\dv[2]{s}{\theta}- \frac{(2n+3)\cos\theta }{\sin^{n+3}\theta}\dv{s}{\theta}+ \frac{2(n+1)(1+\cos^2\theta)s}{\sin^{n+4}\theta}.
\end{align*}
It is easily shown that if $s$ and its derivatives have the asymptotic behaviour given by (3), then by Lemma \ref{lem: L2 sufficent boundary growth} the asymptotic fall off of $\mathscr{L}^n_n \left(s/\sin^{n+2}\theta \right)$ is sufficient for $\mathscr{L}^n_n \left(s/\sin^{n+2}\theta \right)\in \lsin$. To prove (2) we have

\[
\frac{s}{\sin^2\theta}=\frac{s}{\sin^m\theta}\cdot \sin^{m-2}\theta \to 0 \enskip \text{ as } \theta\to 0,\pi,
\]
since $m>n+3>2$ for $n\in(-1,1)$.
\end{proof}
\vspace{0.1in}

\subsection{The Eigenbasis Expansion}
\vspace{0.1in}

\begin{Prop}\label{prop: u decomposition}
If $u\in\lsin$ and $n\in(-1,1)$, then $u$ can be decomposed in $\lsin$ as
\begin{equation} \label{eq: u decomposition}
u=\gamma_{0,n}\sin^{n}\theta+\sum\limits_{m=1}^\infty \gamma_{m,n} \tp^{-n}_{n+m}(\cos \theta).
\end{equation}
\end{Prop}
\begin{proof}
Take $\mathscr{L}^n_n$ to be the Legendre operator with the self-adjoint domain $\dsa$ given by Proposition \ref{prop: Self-Adjoint Domain for L^n_n}. By Propositions \ref{prop: compact res, simple eigval} and \ref{prop: L^n_n is LC} and the Spectral Theorem \ref{thm: spectral theorem}, $\lsin$ has a complete orthonormal basis consisting of $\mathscr{L}^n_n$'s eigenfunctions and kernel. The kernel and eigenspaces of $\mathscr{L}^n_n$ are spanned by the following linearly independent pairs of basis functions:
\begin{center}
\begin{tabular}{ |c|c|c| } 
  \hline
  &$n=0$&$0<|n|<1$\\
  \hline
 Kernel & $\big\{\tq_0( \cos \theta),1\big\}$ & $\big\{\tp^{n}_n( \cos \theta),\sin^n\theta\big\}$ \\ 
  \hline
  Eigenspaces& $\big\{\tq_{\nu}(\cos \theta),\tp_{\nu}( \cos \theta)\big\}$ & $\big\{\tp^n_{\nu}(\cos \theta),\tp^{-n}_{\nu}( \cos \theta)\big\}$ \\
  \hline
\end{tabular}
\end{center}
We will show that $\{\tp^{-n}_{n+m}(\cos\theta)\}_{m=0}^\infty$ are the only functions from the above list belonging to $\dsa$, i.e. the only functions that satisfy the boundary conditions (\ref{eq: seperated boundary conditions for dsa.}). The function $\sin^n\theta$ and the constant function $1$ are easily seen to satisfy the boundary conditions. For the other functions, we use the derivative formula \cite[p362]{olver10}
\[
 \sin\theta\dv[]{\tp^\mu_\nu(\cos \theta)}{\theta}=(1-\mu+\nu)\tp^\mu_{\nu+1}(\cos \theta)-(\nu+1)\cos\theta \tp^\mu_\nu(\cos \theta), \enskip \forall \mu,\nu\in\mathbb{R},
\]
to write the boundary condition (\ref{eq: seperated boundary conditions for dsa.}) with $u=\tp^\mu_\nu(\cos \theta)$ as
\begin{align} \label{eq: boundary condition in terms of P or Q}
\sin^{2n+1}\theta \dv[]{}{\theta}\left(\frac{\tp^\mu_\nu(\cos \theta)}{\sin^n \theta}\right)=(1-\mu+\nu)\sin^n\theta \tp^\mu_{\nu+1}(\cos\theta)-(1+n+\nu)\cos\theta\sin^n\theta \tp^\mu_\nu(\cos\theta).
\end{align}

At $\theta=0$, the asymptotic behavior of $\tp^\mu_\nu(\cos \theta)$ is given by \cite[p361]{olver10}
\begin{equation}\label{eq: asmptotic for P around 0}
\tp^\mu_\nu(\cos\theta) \sim \frac{1}{\Gamma(1-\mu)}\left(\frac{2}{\sin \theta}\right)^\mu, \enskip \mu\neq1,2,3,\ldots
\end{equation}
It follows from equation (\ref{eq: boundary condition in terms of P or Q}) that near $\theta=0$ the asymptotic behavior of the boundary term is
\[
\sin^{2n+1}\theta \dv[]{}{\theta}\left(\frac{\tp^\mu_\nu(\cos \theta)}{\sin^n \theta}\right)\sim -\frac{2^\mu(n+\mu)}{\Gamma(1-\mu)}.
\]
 Hence $\tp^{\mu}_{\nu}(\cos \theta)\in \dsa$ only if $\mu=-n$. A similar argument shows that $Q_\nu(\cos \theta)$ must be rejected from the basis because of it's bad asymptotic behaviour around $\theta=0$ (c.f. \cite[p361]{olver10}). Hence the remaining eigenfunctions are
\begin{center}
\begin{tabular}{ |c|c|c| } 
  \hline
  &$n=0$&$0<|n|<1$\\
  \hline
 Kernel & $1$ & $\sin^n\theta$ \\ 
  \hline
  Eigenspaces & $\tp_{\nu}( \cos \theta)$ & $\tp^{-n}_{\nu}( \cos \theta)$ \\
  \hline
\end{tabular}
\end{center}

We now use the boundary condition at $\theta=\pi$ to determine the possible values of $\nu$. We consider the case $n\neq0$ first. In this case using together the hyper-geometric series representation for $\tp^{-n}_\nu(\cos\theta)$ \cite[p353]{olver10}
\begin{equation}
\tp_{\nu}^{-n}(\cos\theta)={\textstyle{\frac{1}{\Gamma(1+n)}}}\left(\frac{1+\cos\theta}{1-\cos\theta}\right)^{-n / 2} {_2F_1}\left(\nu+1,-\nu ; 1+n ; \frac{1}{2}-\frac{1}{2} \cos\theta\right),
\end{equation}
and the linear transformation \cite[p390]{olver10}
\begin{align*}
{ }_{2} F_{1}(a, b ; c, z)=& {\textstyle{\frac{\Gamma(c) \Gamma(c-a-b)}{\Gamma(c-a) \Gamma(c-b)}}} {_2F_1}(a, b ; a+b+1-c ; 1-z) \\
&+{\textstyle{\frac{\Gamma(c) \Gamma(a+b-c)}{\Gamma(a) \Gamma(b)}}}(1-z)^{c-a-b}{ }_{2} F_{1}(c-a, c-b ; 1+c-a-b ; 1-z),
\end{align*}
one gets the asymptotic behavior of $\tp^{-n}_\nu(\cos\theta)$ around $\theta=\pi$:
\begin{equation}\label{eq: asymptotic for P around pi}
    \tp_{\nu}^{-n}(\cos \theta) \sim \frac{2^{n}\Gamma(n) }{\Gamma(n-\nu) \Gamma(1+n+\nu)} \sin ^{-n} \theta+\frac{\Gamma(-n)}{2^{n} \Gamma(1+\nu) \Gamma(-\nu)} \sin ^{n} \theta.
\end{equation}
We remark that we have $1/\Gamma(z)=0$ whenever $z$ is a negative integer. The asymptotic behavior of the boundary term around $\theta=\pi$ then follows from equations (\ref{eq: boundary condition in terms of P or Q}) and (\ref{eq: asymptotic for P around pi}):
\[
\sin^{2n+1}\theta \dv[]{}{\theta}\left(\frac{\tp^{-n}_\nu(\cos \theta)}{\sin^n \theta}\right)\sim \frac{2^n(1+\nu+n)\Gamma(n)}{\Gamma(\nu+n+2)\Gamma(n-\nu-1)}+\frac{2^n(1+\nu+n)\Gamma(n)}{\Gamma(\nu+n+1)\Gamma(n-\nu)}.
\]
To satisfy the boundary condition we require the above terms to be $0$ which is only the case when $\nu=n+m$ for $m\in\mathbb{N}\cup\{0\}$, so that $1/\Gamma(n-\nu-1)=1/\Gamma(n-\nu)=0$. The $n=0$ case is again similar, we derive the asymptotic behavior of $\tp_\nu(\cos\theta)$ around $\theta=\pi$ and find that $\nu \in \mathbb{N}$ in order for the boundary term to vanish. The family $\{\tp^{-n}_{n+m}(\cos\theta)\}_{m=0}^\infty$ are thus the only family of eigenfunctions in $\dsa$ and form an orthogonal basis of $\lsin$, therefore if $u\in \lsin$
\[
u=\gamma_{0,n}\sin^{n}\theta+\sum\limits_{m=1}^\infty \gamma_{m,n} \tp^{-n}_{n+m}(\cos \theta),
\]
for expansion coefficients $\{\gamma_{m,n}\}_{m=0}^\infty$.
\end{proof}
\vspace{0.1in}

For well-behaved surfaces, Proposition \ref{prop: u decomposition} gives an expansion of the surfaces astigmatism in terms of the Legendre functions.
\vspace{0.1in}

\begin{Thm}\label{thm: decomposition of geometric quantities}
Let $\mathcal{S}$ be a $C^2$-smooth rotationally symmetric, strictly convex sphere with astigmatism satisfying $s/\sin^{n+2}\theta\in \lsin$. If $n\in(-1,1)$, the following geometric quantities associated to $\mathcal{S}$ decompose as:
\begin{equation*} \label{eq: shear decomposition}
s=\gamma_{0,n}\sin^{2n+2}\theta+\sin^{n+2}\theta \sum\limits_{m=1}^\infty \gamma_{n,m} \tp^{-n}_{n+m}(\cos \theta),
\end{equation*}
\begin{equation*}\label{eq: r1 decomposition}
r_1=C_1+\frac{\gamma_{0,n}}{2(n+1)}\sin^{2n+2}\theta+\sin^{n+2}\theta\sum\limits_{m=1}^\infty \gamma_{m,n} \left\{\tp^{-(n+2)}_{n+m}(\cos \theta)+\cot \theta\tp^{-(n+1)}_{n+m}(\cos \theta)\right\},
\end{equation*}
\begin{equation*}\label{eq: support function decomposition}
r=C_2\cos\theta+C_1+\gamma_{0,n}\left[\frac{\sin^{2n+2}\theta}{2(n+1)}-\cos\theta\int^\theta_0\sin^{2n+1}\theta d\theta\right]+\sin^{n+2}\theta\sum\limits_{m=1}^\infty \gamma_{m,n} \tp^{-(n+2)}_{n+m}(\cos \theta),
\end{equation*}
where $C_1=r_1(0)$, $C_2=r(0)-r_1(0)$ and the first equality is to be understood as an equality in $\lsin$ where as the last two are point-wise.
\end{Thm}
\begin{proof}
The expansion for $s$ follows from Proposition \ref{prop: u decomposition} by letting $s=\sin^{n+2}\theta \cdot u$. The expansion for $r_1$ can be derived by inserting the expansion of $s$ into the integrated derived Codazzi-Mainardi equation (\ref{eq: codazzi mainardi}):
\begin{align*}
r_1(\theta)&=r_1(0)+\int^\theta_0 \frac{s \cos\theta}{\sin\theta} d\theta\\
&=r_1(0)+\frac{\gamma_{0,n}}{2(n+1)}\sin^{2n+2}\theta+\sum\limits_{m=1}^\infty \gamma_{m,n}\int^\theta_0\cos \theta\sin^{n+1}\theta\tp^{-n}_{n+m}(\cos \theta)d\theta.
\end{align*}
We then use the standard integral \cite[p368)]{olver10}
\[
 \int\sin^{\mu+1}\theta\cdot  \tp^{-\mu}_{\nu}(\cos \theta)d \theta=\sin^{\mu+1}\theta\cdot  \tp^{-(\mu+1)}_{\nu}(\cos \theta) \enskip \forall \mu,\nu \in \mathbb{R},
\]
to evaluate the following by parts:
\[
\int^\theta_0\cos \theta\sin^{n+1}\theta\tp^{-n}_{n+m}(\cos \theta)d\theta=\sin^{n+2}\theta\tp^{-(n+2)}_{n+m}(\cos \theta)+\cos\theta\sin^{n+1}\theta\tp^{-(n+1)}_{n+m}(\cos \theta),
\]
where the boundary term at $\theta=0$ vanishes because of the asymptotic behaviour of $\tp^{-(n+1)}_{n+m}$, see equation (\ref{eq: asmptotic for P around 0}). This completes the derivation for $r_1$, to derive the expansion of $r$ one may either insert the astigmatism decomposition into equation (\ref{eq: support function from quadratures}) and proceed as above, or we can integrate the decomposition of $r_1$ by virtue of equation (\ref{eq: curvatures in terms of r}). The calculation is analogous.
\end{proof}
\vspace{0.1in}

For a given rotationally symmetric surface $\mathcal{S}$ with astigmatism $s$, it is intriguing to ask in what way does the expansion of $s$ given in Theorem \ref{thm: decomposition of geometric quantities} convey geometric information about $\mathcal{S}$?
\vspace{0.1in}

\begin{Cor}
Let $\gamma_{m,n}$ be the expansion coefficients as given in Theorem \ref{thm: decomposition of geometric quantities} for $n\in(-1,1)$. We have the following equalities
\begin{align*}
&\gamma_{0,n}=\frac{\Gamma(n+\frac{3}{2})}{\sqrt{\pi}\Gamma(n+1)}\cdot (f_0-f_\pi), & \gamma_{1,n}=\frac{\Gamma(2n+4)}{2^{n+1}\Gamma(n+1)}\cdot (r_1(\pi)-r_1(0)).
\end{align*}
hence $\gamma_{0,n}$ characterises the distance between the focal points of $\mathcal{S}$ while $\gamma_{1,n}$ characterises the differences between the radii of curvature at each pole.
\end{Cor}
\begin{proof}
Starting with the expansion formula for $s$ in Theorem \ref{thm: decomposition of geometric quantities}, we divide by $\sin \theta$ and integrate to find the following expression satisfied by $\gamma_{0,n}$:
\[
\int^\pi_0\frac{s}{\sin\theta} d\theta=\gamma_{0,n}\int^\pi_0\sin^{2n+1}\theta d\theta,
\]
where the higher order terms in the expansion vanish via orthogonality. The integral on the left hand side can be evaluated by using sequentially, equation (\ref{eq: codazzi mainardi}), integration by parts, equation (\ref{eq: curvatures in terms of r}) and Remark \ref{rem: focal point positions}:
\begin{align*}
    \int^\pi_0\frac{s}{\sin\theta} d\theta &=f_0-f_\pi,
\end{align*}
which proves the claim for $\gamma_{0,n}$ once the remaining integral on the right hand side is calculated. The claim concerning $\gamma_{1,n}$ follows by letting $\theta = \pi$ in the expansion of $r_1$. All higher order terms are zero apart from the $m=1$ term (by equation (\ref{eq: asymptotic for P around pi})) and we have
\[
r_1(\pi)=r_1(0)+\left.\gamma_{1,n}\left\{\sin^{n+2}\theta \tp^{-(n+2)}_{n+1}(\cos \theta)+\sin^{n+1}\theta\cos \theta\tp^{-(n+1)}_{n+1}(\cos \theta)\right\}\right|_{\theta=\pi}.
\]
The $m=1$ term satisfies

\[
\sin^{n+2}\theta \tp^{-(n+2)}_{n+1}(\cos \theta)+\sin^{n+1}\theta\cos \theta\tp^{-(n+1)}_{n+1}(\cos \theta) \to \frac{\Gamma(n+1)2^{n+2}}{\Gamma(2n+4)},
\]
as $\theta\to\pi$ (again by equation (\ref{eq: asymptotic for P around pi})), which completes the proof.
\end{proof}
\vspace{0.1in}

\subsection{Time Evolution of the Eigenbasis Expansion}\hspace{1cm}\\

In this section we give a solution for each $n\in(-1,1)$ to the time evolution problem (\ref{eq: time evolution problem in lsin}) for all $t\geq0$ and for almost every $\theta \in [0,\pi]$. Furthermore this solution is unique in the sense that it is equal point-wise to a strong solution of (\ref{eq: time evolution problem in lsin}) whenever a strong solution exists, for all most all $\theta\in[0,\pi]$, for every $t \geq 0$.

\vspace{0.1in}

\begin{Prop}\label{prop: time dep decomposition of u}
Fix $n\in(-1,1)$ and let $\{e_m\}^\infty_{m=1}$ be the Legendre functions $\{\tp^{-n}_{n+m}(\cos\theta)\}^\infty_{m=0}$, normalised with respect to the $\lsin$ inner product, i.e.
\[
e_m=\frac{\tp^{-n}_{n+m}(\cos\theta)}{||\tp^{-n}_{n+m}(\cos\theta)||}.
\]
Let $u_0\in\dsa$ have the eigenbasis decomposition
\[
u_0=\sum^\infty_{m=0}a_m e_m.
\]
and $u:[0,\infty) \to \lsin$ be the mapping defined by the series
\begin{align*}
& u(t)=\sum^\infty_{m=0}A_m(t) e_m, & A_m(t)=a_m \exp \{-(2n+1+m)mb\; t\}.
\end{align*}
Then, for each $t\geq0$, $u(t)$ converges in $\lsin$, is an element of $\dsa$ and
\begin{align}\label{eq: ptwise u(t)}
 [u(t)](\theta)=\sum^\infty_{m=0}A_m(t) e_m(\theta),
\end{align}
point-wise for almost all $\theta \in(0,\pi)$ (with respect to the Lebesgue measure). 

Furthermore, $u(t)$ is the unique classical solution to equation (\ref{eq: time evolution problem in lsin}) for almost all $\theta \in (0,\pi)$, $t \in (0,\infty)$ with initial data $u_0$.

\end{Prop}
\begin{proof}
First we remark that $u(0)=u_0$. Take $t\in[0,\infty)$ and let ${\lambda_{m}=-(2n+1+m)m}$, which is the eigenvalue of $\mathscr{L}^n_n$ associated to $e_m$. Since $\lambda_{m}\leq0$ for all $m\in\mathbb{N}$, $n\in(-1,1)$ and $b>0$, we have that $A_m(t)$ is decreasing in $t$, and in particular $|A_m(t)|\leq |a_m|$. It follows that since $u_0\in \lsin$, by comparison $u(t)\in\lsin$. 

Furthermore, by the Spectral Theorem \ref{thm: spectral theorem}, $\dsa$ is characterised by those elements $x\in \lsin$ such that $(\lambda_n\langle x, e_n \rangle )_n \in \ell^2$, therefore since $u_0\in\dsa$,
\begin{align*}
     \sum^\infty_{m=0}|\lambda_{m} A_m(t)|^2 \leq \sum^\infty_{m=0}|\lambda_{m} a_m|^2<\infty,
\end{align*}
and $u(t)\in\dsa$. Therefore $u(t)$ satisfies the boundary conditions of equation (\ref{eq: time evolution problem in lsin}).

We now show $u(t)$ is a classical solution to equation (\ref{eq: time evolution problem in lsin}) a.e. First we show (\ref{eq: ptwise u(t)}) holds, to do so we invoke the following theorem.
\begin{Thm}[Rademacher-Menchov]\cite[p251]{kashin89}\\
Let $\{e_n\}_n$ be an orthonormal series with respect to $\lsin$. The series
\[
\sum^\infty_{n=1}b_n e_n(\theta),
\]
converges for almost all $\theta\in(0,\pi)$ if the sequence $(b_n)_n$ satisfies
\[
\sum^\infty_{n=0}(\log_2(n+1)b_n)^2 < \infty.
\]
\qed
\end{Thm}
To see that the sequence $(a_m)_m$ satisfies the requirement of Rademacher-Menchov, it is enough to notice that the sequence $(\lambda_m)_m$ grows quadratically with $m$, hence
\[
\sum^\infty_{m=M}(\log_2(m+1)a_m)^2 < \sum^\infty_{m=M}|\lambda_m a_m|^2<\infty
\]
for some $M\in\mathbb{N}$. Therefore by Rademacher-Menchov $\sum^\infty_{m=0}a_m e_m(\theta)$ converges a.e., and so must $\sum^\infty_{m=0}A_m(t) e_m(\theta)$ for all $t\geq0$. Now we have proved convergence, the equality in (\ref{eq: ptwise u(t)}) follows by uniqueness of limits, since the partial sums of $u(t)$ converge a.e. along some sub-sequence to $u(t)$.

Fix $\theta\in[0,\pi]$ such that the $u_0(\theta)$ is equal to the series given in equation (\ref{eq: ptwise u(t)}). By Fubini's Theorem on differentiation \cite[p527]{jones93}, since $A_m(t)$ is monotonic in $t$ for all $m\in\mathbb{N}$ we have from term by term differentiation
\[
\partial_t\big([u(t)](\theta)\big)=\sum^\infty_{m=0}\lambda_m A_m(t) e_m(\theta)
\]
a.e. $t\in[0,\infty)$ (with respect to the Lebesgue measure). However this means we can find arbitrarily small $t'$ such that the above series converges and therefore converges for all $t>t'$, i.e. we have 
\begin{align*}
&\partial_t\big([u(t)](\theta)\big)=\sum^\infty_{m=0}\lambda_m A_m(t) e_m(\theta)
&\forall t\in (0,\infty),\text{ a.e. }\theta \in(0,\pi).
\end{align*}
However since $u(t)\in\dsa$, we have $\mathscr{L}^n_nu(t)=\sum^\infty_{m=0}\lambda_mA_m(t)e_n$, and because the series $\sum^\infty_{m=0}\lambda_m A_m(t) e_m(\theta)$ converges a.e. with respect to $\theta$, we have
\begin{align*}
\left[\mathscr{L}^n_nu(t)\right](\theta)&=\sum^\infty_{m=0}\lambda_mA_m(t)e_m(\theta) && \forall t\in [0,\infty),\text{ a.e. }\theta \in(0,\pi),\\
&=\partial_t\big([u(t)](\theta)\big) && \forall t\in (0,\infty),\text{ a.e. }\theta \in(0,\pi).
\end{align*}
Hence $u(t)$ satisfies the desired time evolution equation.

Uniqueness follows from the theory of ODEs by projecting the time evolution equation into each eigenbasis.

\end{proof}

Making the substitution $u=s/\sin^{n+2}\theta$ turns the time evolution problem (\ref{eq: time evolution problem in lsin}) into 
\begin{align}\label{eq: time evolution problem for astig in Lsin}
    \begin{cases}
    \displaystyle{\pdv[]{}{t}\left(\frac{s}{\sin^{n+2}\theta}\right)=b \cdot \mathscr{L}^n_n \left(\frac{s}{\sin^{n+2}\theta}\right)}, & [0,\infty) \times [0,\pi]\\
    \displaystyle{\frac{s(t,\cdot)}{\sin^{n+2}(\cdot)}\in \dsa} & t\in [0,\infty)\\
  s=s_0 & \{t=0\} \times [0,\pi],
    \end{cases}
\end{align}
which describes the evolution of a surfaces astigmatism under the linear Hopf flow (\ref{eq: s/sin evolution}).
\vspace{0.1in}

\begin{Cor}
Let $\mathcal{S}_0 \in \mathscr{W}$ be an initial surface with an astigmatism $s_0$ satisfying $s_0/\sin^{n+2}\theta \in \dsa$. Under the linear Hopf flow (\ref{eq: curvature flow equation}) and (\ref{eq: quantised slope equ}) with $n\in(-1,1)$, for all $\theta\in[0,\pi]$, $t\geq0$, the following quantities evolve as
\begin{align*}
& s(t,\theta)=\gamma_{0,n}(t)\sin^{2n+2}\theta+\sum\limits_{m=1}^\infty \Gamma_{n,m}(t) \sin^{n+2}\theta \tp^{-n}_{n+m}(\cos \theta), & \big(\text{a.e. }\theta\in[0,\pi]\big),
\end{align*}

\begin{equation*}
r_1(t,\theta)=C_1(t)+\frac{\gamma_{0,n}\sin^{2n+2}\theta}{2n+2}+\sin^{n+2}\theta \sum\limits_{m=1}^\infty \Gamma_{n,m}(t) \left[\tp^{-(n+2)}_{n+m}(\cos \theta)+\cot \theta\tp^{-(n+1)}_{n+m}(\cos \theta)\right],
\end{equation*}

\begin{multline*}
r(t,\theta)=C_2(t)\cos\theta+C_1(t)+\gamma_{0,n}\left[\frac{\sin^{2n+2}\theta}{2n+2}-\cos\theta\int^\theta_0\sin^{2n+1}\theta d\theta\right]\\+\sin^{n+2}\theta\sum\limits_{m=1}^\infty \Gamma_{n,m}(t) \tp^{-(n+2)}_{n+m}(\cos \theta),
\end{multline*}
for all $\theta \in [0,\pi]$, where $
    \Gamma_{m,n}(t)=\gamma_{m,n}\exp\big\{-(2n+1+m)mb\cdot t\big\}$
and $\gamma_{m,n}$ are the decomposition coefficients of the initial surface $\mathcal{S}_0$ in the basis $\{\tp^{-n}_{n+m}(\cos\theta)\}^\infty_{m=0}$. Furthermore the constants $C_1(t)$ and $C_2(t)$ evolve as
\begin{align*}
    &C_1(t)=C_1(0)e^{-2(n+1)bt}+\frac{c\left(1-e^{-2(n+1)bt}\right)}{2(n+1)b}, &C_2=\text{ constant }.
\end{align*}
\end{Cor}
\begin{proof}
Fix $n\in(-1,1)$ and let $s=\sin^{n+2}\theta\cdot u$. Expand $u$ as in Proposition \ref{prop: time dep decomposition of u}. The coefficients $\Gamma_{m,n}(t)$ are related to $A_m(t)$ by $A_m(t)=\Gamma_{n,m}(t) ||\tp^{-n}_{n+m}(\cos\theta)||$. This derives the expression for $s(t,\theta)$.
The expressions for $r_1(t,\theta)$ and $r(t,\theta)$ follow by integration as in the proof of Theorem \ref{thm: decomposition of geometric quantities}.

To derive the behaviour of $C_1(t)$ and $C_2(t)$, we ensure they satisfy equation (\ref{eq: evolution of support function}) for the linear Hopf flow. i.e.
\begin{align*}
\pdv[]{r(t,\theta)}{t}&=ar_1(t,\theta)+br_2(t,\theta)+c\\
&=c+b\left[s(t,\theta)-2(n+1)r_1(t,\theta)\right].
\end{align*}
where we have first written the right hand side of (\ref{eq: evolution of support function}) in terms of $r_1$ \& $r_2$, and then used the identity $-a/b=2n+3$. Substituting the expansions of $r(t,\theta)$, $r_1(t,\theta)$ and $s(t,\theta)$ into this equation and collecting together terms gives us the relationship
\begin{align*}
    \partial_tC_2(t)\cos\theta+\partial_tC_1(t)&=
    c-2(n+1)bC_1(t)+b\sin^{n+2}\theta \sum\limits_{m=1}^\infty \Gamma_{n,m}(t)\Delta_{m,n},
\end{align*}
where
\begin{align*}
    \Delta_{m,n}=\tp^{-n}_{n+m}(\cos \theta)-2(n+1)\cot \theta\tp^{-(n+1)}_{n+m}(\cos \theta)-(2n+2+m)(m-1)\tp^{-(n+2)}_{n+m}(\cos \theta).
\end{align*}

However, by letting $\mu=-(n+2)$ and $v=n+m$ in the following recurrence relation between the Legendre functions \cite[p362]{olver10}
\[
\mathrm{P}_{\nu}^{\mu+2}(x)+2(\mu+1) x\left(1-x^{2}\right)^{-1 / 2} \mathrm{P}_{\nu}^{\mu+1}(x)+(\nu-\mu)(\nu+\mu+1) \mathrm{P}_{\nu}^{\mu}(x)=0,
\]
we can see that $\Delta_{m,n}$ is identically $0$, hence we have the following evolution equation
\[
\partial_tC_2(t)\cos\theta+\partial_tC_1(t)=
    c-2(n+1)bC_1(t).
\]
It is easy to see that $\partial_tC_2(t)=0$. Solving the remaining ODE gives the stated time evolution of $C_1(t)$.
\end{proof}
\begin{remark}
We have the following asymptotic behaviour as $t\to \infty$, 
\[
r(t,\theta)\sim C_2\cos\theta +\frac{c}{2(n+1)b}+\gamma_{0,n}\left[\frac{\sin^{2n+2}\theta}{2n+2}-\cos\theta\int^\theta_0\sin^{2n+1}\theta d\theta\right].
\]
This is the astigmatism of a linear Hopf sphere. If $\gamma_{0,n}=0$, i.e. if the focal points of the initial surface coincide, then the astigmatism is that of a sphere with radius
\[
\frac{c}{2(n+1)b}=-\frac{c}{a+b},
\]
as claimed.
\end{remark}

\vspace{0.2in}

\noindent{\bf Statements and Declarations}:

The second author was supported by the Institute of Technology, Tralee / Munster Technological University Postgraduate Scholarship Programme. 
\vspace{0.2in}


\end{document}